\documentclass{siamltex1213}

\usepackage{amsmath,amssymb,amsfonts,cmmib57,mathdots,xcolor,mathrsfs}
\usepackage[latin1]{inputenc}
\usepackage{enumerate}
\usepackage[linesnumbered,lined,commentsnumbered]{algorithm2e}

\newcommand{\rev}{\ensuremath\mathrm{rev}}

\newtheorem{remark}[theorem]{Remark}
\newtheorem{example}[theorem]{Example}

\title{A note on eigenvector error bounds  for polynomial eigenvalue problems}
\author{Javier P\'{e}rez\footnotemark[1]}
\begin{document}
\maketitle
\slugger{simax}{xxxx}{xx}{x}{x--x}
\renewcommand{\thefootnote}{\fnsymbol{footnote}}
\footnotetext[1]{Department of Mathematical Sciences, University of Montana, USA. Email: {\tt javier.perez-alvaro@mso.umt.edu}.}

\renewcommand{\thefootnote}{\arabic{footnote}}

\begin{abstract}
The standard approach for finding the eigenvalues and the eigenvectors of a matrix polynomial starts by embedding the coefficients of the  polynomial into a matrix pencil, known as linearization.
Building on the pioneering work of Nakatsukasa and Tisseur, we present in this work error bounds for computed eigenvectors of matrix polynomials.
Our error bounds are  applicable to any linearization satisfying two properties.
First, eigenvectors of the original matrix polynomial can be recovered from those of the linearization without any arithmetic computation.
Second, the linearization presents one-sided factorizations, which relate the residual for the linearization with the residual for the polynomial.
Linearizations satisfying these two properties include the family of block Kronecker linearizations.
The error bounds imply that an eigenvector has been computed accurately when the residual norm is small, provided that the computed associated eigenvalue is well-separated from the rest of the spectrum of the linearization.
The theory is illustrated by numerical examples.
\end{abstract}

\pagestyle{myheadings}
\thispagestyle{plain}

\section{Introduction}\label{sec:intro}
The \emph{polynomial eigenvalue problem} (hereafter PEP) associated with a regular\footnote{A matrix polynomial $P(\lambda)$ is regular if $\det P(\lambda)\not\equiv 0$.} \emph{matrix polynomial}
\begin{equation}\label{eq:poly}
P(\lambda) = \lambda^d A_d+\lambda^{d-1} A_{d-1}+\cdots+A_1\lambda + A_0,
\end{equation}
where the coefficients $A_i$ are $n\times n$ complex or real matrices, consists in finding scalars $\lambda$ and vectors $x$ such that
\begin{equation}\label{eq:PEP}
P(\lambda)x = 0.
\end{equation}
The scalar $\lambda$ is an \emph{eigenvalue} of $P(\lambda)$, and $x$ is the corresponding \emph{eigenvector}.
The pair $(\lambda,x)$ is known as an \emph{eigenpair} of $P(\lambda)$. 
Solving PEPs is an important task in scientific computation \cite{NEP}, and has received considerable attention in the last decades (see \cite[Chapter 12]{Volker} and the references therein).

The most common approach for solving the PEP \eqref{eq:PEP} starts by embedding the matrix coefficients $A_i$ into the coefficients of a larger matrix pencil\footnote{In this work, matrix pencil and pencil refer to a matrix polynomial of degree at most 1.} $L(\lambda)=A-\lambda B$, which has the same eigenvalues as the matrix polynomial $P(\lambda)$.
This matrix pencil is called a  \emph{linearization}.
The linearization replaces the PEP \eqref{eq:PEP} by a \emph{generalized eigenvalue problem} (GEP)
\[
(A-\lambda B)v = 0.
\]
So, the eigenvalues of the polynomial $P(\lambda)$  can  be obtained from the generalized Schur decomposition of the linearization \cite{QZ}
\[
Q(A-\lambda B)Z  = T_A - \lambda T_B,
\]
where $T_A$ and $T_B$ are upper triangular matrices, and $Q$ and $Z$ are unitary.
After the eigenvalues are calculated, the eigenvector of the linearization are computed by running a few steps of the inverse iteration.
In principle,  the eigenvectors of $P(\lambda)$ and the eigenvectors of the linearization $L(\lambda)$ need not be related in a simple way.
However, most linearizations  allow us to recover easily the eigenvectors of $P(\lambda)$ from those of $L(\lambda)$.
This approach for solving PEPs is followed, for instance, by the {\tt polyeig} command in MATLAB \cite{Matlab}.

Any numerical algorithm (based or not on linearization) for solving PEPs is affected by round-off errors due to the limitations of floating point arithmetic.
If $(\lambda,x)$ and $(\widetilde{\lambda},\widetilde{x})$ denote, respectively, exact and  computed eigenpairs of $P(\lambda)$, a main concern is quantifying the errors in the computed solutions:
\[
\frac{|\lambda-\widetilde{\lambda}|}{|\lambda|} \quad \quad \mbox{and} \quad \quad  \sin\angle(x,\widetilde{x})
\]
where  $\angle(u,w)$ denotes the acute angle between vectors $u$ and $v$ (see Section \ref{sec:acute angle} for its precise definition).

Upper bounds for the eigenvalue relative errors $|\lambda-\widetilde{\lambda}|/|\lambda|$ have been obtained  for linearizations in the vector space $\mathbb{DL}(P)$ in \cite{conditioning}, and for linearizations in the family of block Kronecker pencils  in \cite{conditioning Kronecker}.
These upper bounds are based on the fact that generalized eigensolvers like the QZ algorithm are backward stable.
The backward stability implies that the computed eigenvalues $\widetilde{\lambda}$ are the exact eigenvalues of 
\[
L(\lambda) + \Delta L(\lambda) = A+\Delta A-\lambda(B+\Delta B),
\]
where $\|\Delta A\|_2 \leq \mathcal{O}(u)\|A\|_2$ and $ \|\Delta B\|_2\leq \mathcal{O}(u)\|B\|_2$. 
Here $u$ denotes the unit round-off, and by $\mathcal{O}(u)$ we denote  any quantity that is upper bounded by $u$ times a modest constant.
Then, to first order in $u$, we have
\begin{equation}\label{eq:eigenvalue error}
\frac{|\lambda-\widetilde{\lambda}|}{|\lambda|}\leq \mathcal{O}\left(\kappa_L(\lambda)u\right) = \frac{\kappa_L(\lambda)}{\kappa_P(\lambda)}\cdot \mathcal{O}\left(\kappa_P(\lambda)u\right),
\end{equation}
where $\kappa_L(\lambda)$ and $\kappa_P(\lambda)$ denote the condition numbers of $\lambda$ as eigenvalue of $L(\lambda)$ and $P(\lambda)$, respectively, \cite{Tisseur}.
The ratio $\kappa_L(\lambda)/\kappa_P(\lambda)$ measures, then, how far is this approach from being forward stable.
Simple conditions for $\kappa_L(\lambda)/\kappa_P(\lambda)$ to be moderate have been identified in \cite{conditioning,conditioning Kronecker}.
Notice that under these conditions, the bound \eqref{eq:eigenvalue error} guarantees that the well-condition eigenvalues of $P(\lambda)$ (i.e., those eigenvalues of $P(\lambda)$ with $\kappa_P(\lambda)\approx 1$) can be computed with high relative accuracy if we apply a backward stable eigenvalue algorithm to the linearization.

Compared to the eigenvalue error bounds in \cite{conditioning,conditioning Kronecker}, not many results on the accuracy of computed eigenvectors---as measure by the acute angle $\angle(x,\widetilde{x})$---exist \cite{DT,MT,TN}.
 The difficulties stem from the fact that a linearization $L(\lambda)$ and the matrix polynomial $P(\lambda)$ do not have the same eigenvectors.
 Nonetheless, the eigenvectors of most of the linearizations introduced in the last years are closely related to those of the polynomial.
Inspired by \cite{TN}, this eigenvector property will allow us to derive simple eigenvector error bounds for many families of linearizations.
As the bounds obtained in \cite{TN}, our error bounds depend  on two quantities, the residual norm $\|P(\widetilde{\lambda})\widetilde{x}\|_2$, and the separation of $\widetilde{\lambda}$ to the eigenvalues of $L(\lambda)$ other than $\lambda$.
These bounds show that an eigenvector of a matrix polynomial $P(\lambda)$ can be computed accurately, provided that the associated eigenvalue is well-separated from the rest of the spectrum of the linearization.
In contrast to the bounds in \cite{TN}, our error bounds are valid essentially for all the linearizations available in the literature, since they do not required the Vandermonde-like structure of the eigenvectors of the linearization.
  
 We begin with Section \ref{sec:preliminaries} by recalling some basic concepts that are used throughout
the paper, followed in Section \ref{sec:Kronecker linearizations} by the  definition of block Kronecker linearizations. 
Section \ref{sec:error-bounds} derives error bounds for computed eigenvectors of matrix polynomials.
Then, Section \ref{sec:error-bounds-BKL} shows that the new error bounds generalize the error bound by Nakatsukasa and Tisseur \cite{TN} for the Frobenius companion form to the more general construction of block Kronecker linearization.  
Section \ref{sec:experiments} illustrates the main results by numerical examples. 

Throughout the paper we use the following notation.
 We denote by $I_n$ the $n\times n$ identity matrix, and by $0$ the matrix with all its entries equal to zero, whose  size should be clear from the context.
By $A\otimes B$, we denote the Kronecker product of two matrices $A$ and $B$. 
We denote by $\mathbb{C}[\lambda]$ the ring of polynomials in the variable $\lambda$ with complex coefficients.
 The set of $m \times n$ matrix polynomials, this is, the set of $m\times n$ matrices with entries in $\mathbb{C}[\lambda]$, is denoted by $\mathbb{C}[\lambda]^{m\times n}$.
We say that the polynomial $P(\lambda)$ in \eqref{eq:poly} has \emph{degree} $d$, when $A_d\neq 0$. 
Otherwise, we say that $P(\lambda)$ has \emph{grade} $d$.

\section{Preliminaries}\label{sec:preliminaries}

\subsection{Strong linearizations of matrix polynomials}\label{sec:Matrix Polynomials}


We begin by recalling the definition of strong linearization.
A \emph{linearization} of a matrix polynomial $P(\lambda)$ is a matrix pencil $L(\lambda)=A-\lambda B$ such that
\[
U(\lambda)(A-\lambda B)V(\lambda) = 
\begin{bmatrix}
P(\lambda) & 0 \\ 0 & I_{(d-1)n}
\end{bmatrix}
\]
for some unimodular matrices $U(\lambda)$ and $V(\lambda)$, i.e., $\det(U(\lambda))$ and $\det(V(\lambda))$ are nonzero constants independent of $\lambda$.
If a linearization $L(\lambda) = A-\lambda B$ of $P(\lambda)$ satisfies additionally
\[
W(\lambda)(\lambda A -B)Z(\lambda) = 
\begin{bmatrix}
\rev P(\lambda) & 0 \\ 0 & I_{(d-1)n}
\end{bmatrix},
\]
where $\rev P(\lambda):= \lambda^d A_0+\lambda^{d-1}A_1+\cdots + \lambda A_{d-1}+A_d$, for some unimodular matrices $W(\lambda)$ and $Z(\lambda)$, the linearization $L(\lambda)$ is said to be a \emph{strong linearization} of $P(\lambda)$.
We recall that the definition of strong linearization implies that  $P(\lambda)$ and $L(\lambda)$ have the same \emph{finite and infinite eigenvalues}, with the same algebraic and geometric multiplicities \cite{singular,Lancaster}.

The prototype strong linearization of a matrix polynomial $P(\lambda)$ as in \eqref{eq:poly} is the \emph{Frobenius companion linearization}
\begin{equation}\label{eq:Frob}
\begin{bmatrix}
\lambda A_d+A_{d-1} & A_{d-2} & \cdots & A_1 & A_0 \\
-I_n & \lambda I_n & 0 & \cdots & 0 \\
0 & -I_n & \lambda I_n & \ddots  & \vdots \\
\vdots & \ddots & \ddots & \ddots & 0 \\
0 & \cdots & 0 & -I_n & \lambda I_n
\end{bmatrix},
\end{equation}
which is the linearization used by the MATLAB command {\tt polyeig} \cite{Matlab}, and by the fast and stable polynomial eigenvalue solver \cite{fast}.
Many others linearizations have been introduced in the last decade (see, for example, \cite{Fiedler,kronecker,4m-vspace}) for different reasons, e.g., preservation of algebraic structures \cite{structured kronecker,GoodVibrations}, developing new robust and stable algorithms \cite{NLEIGS,TS-CORK,CORK AAA,Leo2,CORK}, linearizing polynomials in bases other than the monomials \cite{ACL09,Cheby1,Newton,Cheby2,Leo1}, etc.

All the linearizations known by the authors satisfy local right-sided factorizations, which are introduced in Definition \ref{def:local fact}.
\begin{definition}[local right-sided factorization]\label{def:local fact}
A linearization $L(\lambda)$ of a matrix polynomial $P(\lambda)$ is said to present a  \emph{right-sided factorization at $\Omega\subseteq \mathbb{C}$} if
\[
L(\lambda)H(\lambda) = g(\lambda)\otimes P(\lambda),
\]
for some matrix-valued function $H(\lambda)$ and some vector-valued function $g(\lambda)$, where for each $\lambda$ in $\Omega$, $H(\lambda)$ has full column rank and $g(\lambda)$ is nonzero.
The matrix-valued function  $H(\lambda)$ is called a \emph{right-sided factor at $\Omega$}.
\end{definition}

The right-sided factorizations in \cite[Equation (2.4)]{framework} are local right-sided factorizations, but not the other way around, since we do not require the entries of $H(\lambda)$ or $g(\lambda)$ to be polynomials in $\lambda$.

\begin{example} 
Consider for illustrative purposes the Frobenius companion form \eqref{eq:Frob} associated with a matrix polynomial $P(\lambda)$ as in \eqref{eq:poly} of degree $d=3$. 
We can easily check
\begin{equation}\label{eq:ex-factorization1}
\begin{bmatrix}
\lambda A_3 + A_2 & A_1 & A_0 \\
-I_n & \lambda I_n & 0 \\
0 & -I_n & \lambda I_n 
\end{bmatrix}
\begin{bmatrix}
\lambda^2 I_n \\ \lambda I_n \\ I_n
\end{bmatrix} =
\begin{bmatrix}
1 \\ 0 \\ 0
\end{bmatrix}\otimes P(\lambda).
\end{equation}
Notice that $\begin{bmatrix} \lambda^2 I_n & \lambda I_n & I_n \end{bmatrix}^T$ has full column rank at every $\lambda\in\mathbb{C}$.
Hence, \eqref{eq:ex-factorization1} is a right-sided factorization at $\mathbb{C}$.
Now, dividing both sides of \eqref{eq:ex-factorization1} by $\lambda^2$, yields
\begin{equation}\label{eq:ex-factorization2}
\begin{bmatrix}
\lambda A_3 + A_2 & A_1 & A_0 \\
-I_n & \lambda I_n & 0 \\
0 & -I_n & \lambda I_n 
\end{bmatrix}
\begin{bmatrix}
 I_n \\ \lambda^{-1} I_n \\ \lambda^{-2} I_n
\end{bmatrix} =
\begin{bmatrix}
\lambda^{-2} \\ 0 \\ 0
\end{bmatrix}\otimes P(\lambda),
\end{equation}
which is a right-sided factorization at $\mathbb{C}\setminus\{0\}$.
The factorizations \eqref{eq:ex-factorization1} and \eqref{eq:ex-factorization2} are easily generalized to any degree $d$.
\end{example}

A key feature of the eigenvector error bounds obtained in Section \ref{sec:error-bounds} is that they are applicable to any linearization satisfying the following two properties:

\medskip

\begin{itemize}
\item[\rm P1.] For each eigenvalue $\lambda_0$ of $P(\lambda)$, the linearization $\mathcal{L}(\lambda)$ presents a right-sided factorization at $\Omega$ of the form
\begin{equation}\label{eq:one-sided-fact}
\mathcal{L}(\lambda)H(\lambda) = g(\lambda)\otimes P(\lambda),
\end{equation}
for some complex region $\Omega$ that contains $\lambda$, and  some local right-sided factor $H(\lambda)$ at $\Omega$.

\item[\rm P2.] If we partition the local right-sided factor $H(\lambda)$ in \eqref{eq:one-sided-fact} into $d$ blocks of size $n\times n$ each, which we denote by $H_j(\lambda)$, for $j=1,\hdots,d$, then $H_i(\lambda)=I_n$, for some $i\in\{1,\hdots,n\}$.  
\end{itemize}

\medskip

Most linearizations satisfy properties P1 and P2.
Hence, the error bounds in Section \ref{sec:error-bounds} are applicable to a huge plethora of linearizations.
In particular, we will show in Section \ref{sec:Kronecker linearizations} that all the linearizations in the family of block Kronecker pencils satisfy properties P1 and P2.

An important consequence of properties P1 and P2 is that the eigenvectors of a linearization $L(\lambda)$ satisfying P1 and P2 are closely related with those of $P(\lambda)$. 
\begin{theorem}\label{theorem:exact-eigenpair}
Let $P(\lambda)$ be a matrix polynomial, and let $L(\lambda)$ be a linearization of $P(\lambda)$ satisfying properties {\rm P1} and {\rm P2}.
Then, $v$ is an eigenvector $L(\lambda)$ with eigenvalue $\lambda_0$ if and only if $v=H(\lambda_0)x$ for some eigenvector $x$ of $P(\lambda)$ with eigenvalue $\lambda_0$.
\end{theorem}
\begin{proof}
Let $\lambda_0$ be an eigenvalue of $P(\lambda)$, and assume $L(\lambda)$ presents a right-sided factorization at $\Omega$ of the form \eqref{eq:one-sided-fact}, with $\lambda_0\in\Omega$.

Let $x$ be a right eigenvector of $P(\lambda)$ with eigenvalue $\lambda_0$, and let $v=H(\lambda_0)x$.
Property P2 implies that if the vector $x$ is nonzero, then so is $v$.
Evaluating \eqref{eq:one-sided-fact} at the eigenvalue $\lambda_0$, multiplying from the right by $x$, and using $P(\lambda_0)x=0$, yields 
\[
L(\lambda_0)H(\lambda_0)x = L(\lambda_0)v=g(\lambda_0)\otimes (P(\lambda_0)x) = 0.
\]
Hence, $v$ is an eigenvector of $L(\lambda)$ with  eigenvalue $\lambda_0$.

Assume $\lambda_0$ as an eigenvalue of $P(\lambda)$ has geometric multiplicity $m$, and let $v$ be any eigenvector of $L(\lambda)$ associated with $\lambda_0$.
Since $L(\lambda)$ is a linearization of $P(\lambda)$, the multiplicity of $\lambda_0$ as an eigenvalue of $L(\lambda)$ is also $m$. 
Let $x_1,\hdots,x_m$ be linearly independent eigenvectors of $P(\lambda)$ associated with $\lambda_0$, and let $v_i:=H(\lambda_0)x_i$, for $i=1,\hdots,m$. 
Since $H(\lambda_0)$ has full column rank, the vectors $v_1,\hdots,v_m$ are linearly independent.
In other words, the vectors $v_1,\hdots,v_m$ form a basis for the nullspace of $P(\lambda_0)$. 
Hence $v$ must be of the form
\[
v = \sum_{i=0}^m c_i v_i = \sum_{i=0}^m c_i H(\lambda_0)x_i = H(\lambda_0)\left( \sum_{i=0}^m c_ix_i\right) =: H(\lambda_0)x,
\]
as we wanted to show.
\end{proof}

Another important consequence of local right-sided factorizations, is that they allow us to relate residual norms for the linearization $L(\lambda)$ to residual norms for the matrix polynomial $P(\lambda)$.
\begin{theorem}\label{thm:residual norm}
Let $P(\lambda)$ be a matrix polynomial, and let $L(\lambda)$ be a linearization of $P(\lambda)$.
Assume $L(\lambda)$ satisfies a right-sided factorization at $\Omega$ of the form \eqref{eq:one-sided-fact}.
If $\lambda_0\in\Omega$, then
\[
\|P(\lambda_0)x\|_2 = \frac{\|L(\lambda_0)H(\lambda_0)x\|_2}{\|g(\lambda_0)\|_2},
\]
for any vector $x\in\mathbb{C}^n$.
\end{theorem}
\begin{proof}
Multiplying \eqref{eq:one-sided-fact} from the right by the vector $x$ and evaluating at $\lambda=\lambda_0$, gives $L(\lambda_0)H(\lambda_0)x = g(\lambda_0)\otimes (P(\lambda_0)x)$.
Then, taking norms on both sides of the equation, and using $\|A\otimes B\|_2 = \|A\|_2\|B\|_2$, for any matrices $A$ and $B$, the result is readily established.
\end{proof}

\subsection{The acute angle between two vectors}\label{sec:acute angle}

The acute angle between two vectors $u,w\in\mathbb{C}^n$ is given by
\begin{equation}\label{eq:def-angle}
 \angle(u,w) := \mathrm{arccos}\left( \frac{|w^*u|}{\|u\|_2\|w\|_2}\right).
\end{equation}

Lemma \ref{lemma:angle-char}, which can also be found in \cite{q-Ritz}, gives a simple variational characterization of  $\sin \angle(u,w)$.
This is an elementary result, and we present its proof for completeness. 
\begin{lemma}\label{lemma:angle-char}
Let $u,w\in\mathbb{C}^n$. 
Then
\[
\sin\angle(u,w)=\min_{\alpha\in\mathbb{C}}\Big{\|}\frac{u}{\|u\|_2}-\alpha w \Big{\|}_2.
\]
\end{lemma}
\begin{proof}
The minimum of $\| u/\|u\|_2-\alpha w \|_2$ is attained at $\alpha=w^*u/(\|u\|_2\|w\|_2^2)$. 
Hence, notice
\begin{align*}
\Big{\|}\frac{u}{\|u\|_2}- \frac{w^*u}{\|u\|_2\|w\|_2^2} w \Big{\|}_2^2 &= \left( \frac{u^*}{\|u\|_2}- \frac{u^*w}{\|u\|_2\|w\|_2^2} w^* \right) \left( \frac{u}{\|u\|_2}- \frac{w^*u}{\|u\|_2\|w\|_2^2} w \right)\\
&=1 - \frac{|w^*u|^2}{\|u\|_2^2\|w\|_2^2} = 1 - \cos^2\angle(u,w)=\sin^2\angle(u,w),
\end{align*}
as we wanted to show.
\end{proof}

As a consequence of Lemma \ref{lemma:angle-char}, we obtain in Theorem \ref{thm:bound-angle} a result on the sine of the acute angle between two block vectors.
This result is a straightforward generalization of \cite[Lemma 2.3]{q-Ritz}.
\begin{theorem}\label{thm:bound-angle}
Let $u=\left[\begin{smallmatrix} u_1 \\ \vdots \\ u_d \end{smallmatrix}\right]$ and $w=\left[\begin{smallmatrix} w_1 \\ \vdots \\ w_d \end{smallmatrix}\right]$ be two block-vectors with  block-entries $u_j,w_j\in\mathbb{C}^n$, for $j=1,\hdots,d$.
If $\|u_i\|_2=\|w_i\|_2=1$, for some $i\in\{1,\hdots,d\}$, then $\sin \angle (u_i,w_i) \leq \min\{ \|u\|_2,\|w\|_2\} \sin \angle(u,w)$.
\end{theorem}
\begin{proof}
From Lemma \ref{lemma:angle-char}, we have
\begin{align*}
\sin^2\angle (u,w) =& \min_{\alpha\in\mathbb{C}}\Big{\|}\frac{u}{\|u\|_2}-\alpha w \Big{\|}_2^2 =\\ & \min_{\alpha\in\mathbb{C}}\left(\sum_{j\neq i}\Big{\|}\frac{u_j}{\|u\|_2}-\alpha w_j \Big{\|}_2^2 + \Big{\|}\frac{u_i}{\|u\|_2}-\alpha w_i \Big{\|}_2^2\right) \geq \\
 & \min_{\alpha\in\mathbb{C}} \Big{\|}\frac{u_i}{\|u\|_2}-\alpha w_i \Big{\|}_2^2 = \frac{1}{\|u\|_2^2} \sin^2\angle(u_i,w_i).
\end{align*}
Hence, $\sin\angle(u_i,v_i)\leq \|u\|_2\sin\angle(u,v)$.
Since $\sin^2\angle(u,w)=\sin^2\angle (w,u)$, we can analogously  obtain  $\sin\angle(u_i,v_i)\leq \|w\|_2\sin\angle(u,v)$.
Then, the desired result readily follows.
\end{proof}

\subsection{Eigenvector error bounds for generalized eigenvalue problems}\label{sec:GEP}
We recall in Theorem \ref{thm:gap} a well-known result on the sine of the acute angle between a given vector and an eigenvector of a matrix pencil. 
The quantity known as the \emph{separation}, that is,
\begin{equation}\label{eq:def-sep}
\mathrm{sep}\,(\lambda,(A,B)) := \| (A-\lambda B)^{-1} \|_2^{-1}=\sigma_{\rm min}(A-\lambda B),
\end{equation}
plays a key role.
The separation \eqref{eq:def-angle} can be seen as a rough measure of the distance from $\lambda$ to the spectrum of the pencil $A-\lambda B$.
We present the proof of Theorem \ref{thm:gap} for completeness.

\begin{theorem}\label{thm:gap}
Let $L(\lambda)=A-\lambda B$ be a regular matrix pencil, let $(\lambda_0,v)$ be an eigenpair of $L(\lambda)$ and let $(\widetilde{\lambda}_0,\widetilde{v})$ be a given pair, where $\widetilde{\lambda}_0$ is not an eigenvalue of $L(\lambda)$.
Suppose $L(\lambda)$ has a generalized Schur form
\[
Q^*BZ=\begin{bmatrix} \alpha & b^* \\ 0 & B_1 \end{bmatrix} \quad \mbox{and} \quad 
Q^*AZ=\begin{bmatrix} \beta & a^* \\ 0 & A_1 \end{bmatrix} \quad \mbox{with} \quad \lambda_0 = \frac{\beta}{\alpha},
\] 
where $Q=\begin{bmatrix} q_1 & Q_2 \end{bmatrix}$ and $Z=\begin{bmatrix} z_1 & Z_2 \end{bmatrix}$ are unitary with $z_1=v/\|v\|_2$, and  $A_1$ and $B_1$ are upper triangular.
Then
\[
\sin\angle(v,\widetilde{v})\leq \dfrac{\| L(\widetilde{\lambda}_0)\widetilde{v}\|_2}{\|\widetilde{v}\|_2\,\mathrm{sep}\,(\widetilde{\lambda}_0,(A_1,B_1))},
\]
where $\mathrm{sep}(\lambda,(A,B))$ has been defined in \eqref{eq:def-sep}.
\end{theorem}
\begin{proof}
Let
\[
\begin{bmatrix}
\eta \\ \phi
\end{bmatrix} = 
\begin{bmatrix}
z_1^* \\ Z_2^*
\end{bmatrix}\widetilde{v} \quad \mbox{and} \quad
\begin{bmatrix}
\sigma \\ s
\end{bmatrix} =
\begin{bmatrix}
q_1^* \\ Q_2^*
\end{bmatrix}\mathcal{L}(\widetilde{\lambda}_0)\widetilde{v}.
\]
Since $\mathrm{span}\{v\}=\mathrm{span}\{z\}$, we have $\sin\angle(v,\widetilde{v})=\|Z_2^* \widetilde{v}\|_2/\|\widetilde{v}\|_2=\|\phi\|_2/\|\widetilde{v}\|_2$. 
Then, notice
\[
\begin{bmatrix}
\sigma \\ s
\end{bmatrix} =
 Q^*\mathcal{L}(\widetilde{\lambda}_0)Z \begin{bmatrix}
z_1^* \\ Z_2^*
\end{bmatrix}\widetilde{v} = \left( \begin{bmatrix} \beta & a^* \\ 0 & A_1 \end{bmatrix}-\widetilde{\lambda}_0 \begin{bmatrix} \alpha & b^* \\ 0 & B_1 \end{bmatrix} \right)\begin{bmatrix}
\eta \\ \phi
\end{bmatrix},
\] 
which implies $s=(A_1-\widetilde{\lambda}_0 B_1)\phi$.
By assumption, the matrix $A_1-\widetilde{\lambda}_0 B_1$ is nonsingular.
Thus,  $\phi=(A_1- \widetilde{\lambda}_0 B_1)^{-1}s$.
Finally, by the unitary invariance of the spectral norm, we obtain
\begin{align*}
\sin\angle(v,\widetilde{v}) &= \frac{\|\phi\|_2}{\|\widetilde{v}\|_2} \leq \|(A_1-\widetilde{\lambda}_0 B_1)^{-1}\|_2\frac{\|s\|_2}{\|\widetilde{v}\|_2} =
\frac{\|s\|_2}{\|\widetilde{v}\|_2\,\mathrm{sep}\,(\widetilde{\lambda}_0,(A_1,B_1))} \\
&\leq 
\frac{\big{\|} \begin{bmatrix} \sigma \\ s \end{bmatrix}\big{\|_2}}{\|\widetilde{\lambda}_0\|_2\,\mathrm{sep}\,(\widetilde{\lambda}_0,(A_1,B_1))}=
\frac{\| L(\widetilde{\lambda}_0)\widetilde{v} \|_2}{\|\widetilde{v}\|_2\,\mathrm{sep}\,(\widetilde{\lambda}_0,(A_1,B_1))},
\end{align*}
as we wanted to show.
\end{proof}

The quantity $\mathrm{sep}\,(\widetilde{\lambda}_0,(A_1,B_1))$  in Theorem \ref{thm:gap} is usually interpreted as the distance from $\widetilde{\lambda}_0$ to the set of eigenvalues of $L(\lambda)$ other than $\lambda_0$.

\section{Block Kronecker pencils}\label{sec:Kronecker linearizations}

We recall in this section the family of (square) block Kronecker pencils, introduced recently in \cite{kronecker}.

We begin with the auxiliary matrix polynomials that appear repeatedly throughout the following developments:
\[
L_k(\lambda):=\begin{bmatrix}
-1 & \lambda & 0 & \cdots & 0 \\
0 & -1 & \lambda & \ddots & \vdots \\
\vdots & \ddots & \ddots & \ddots & 0 \\
0 & \cdots & 0 & -1 & \lambda
\end{bmatrix}\in\mathbb{C}[\lambda]^{k\times (k+1)} \quad \quad \mbox{and} \quad \quad
\Lambda_k(\lambda):= \begin{bmatrix}
\lambda^k \\ \vdots \\ \lambda \\ 1
\end{bmatrix}.
\]

\begin{definition}[Block Kronecker pencil]
An \emph{$(\epsilon,n,\eta,n)$-block Kronecker pencil}, or, simply, \emph{a block Kronecker pencil}, is a pencil of the form
\begin{equation}\label{eq:block Kronecker pencil}
\mathcal{L}(\lambda)=\left[ \begin{array}{c|c}
M(\lambda) & L_\eta(\lambda)^T\otimes I_n \\ \hline
L_\epsilon(\lambda)\otimes I_n & 0
\end{array}\right],
\end{equation}
where $M(\lambda)$ is arbitrary.
\end{definition}

A block Kronecker pencil \eqref{eq:block Kronecker pencil} is always a strong linearization of a certain $n\times n$ matrix polynomial.
\begin{theorem}[\cite{kronecker}]\label{thm:kronecker pencils}
The block Kronecker pencil \eqref{eq:block Kronecker pencil} is a strong linearization of the matrix polynomial
\begin{equation}\label{eq:Q}
Q(\lambda):=(\Lambda_\eta(\lambda)^T\otimes I_n)M(\lambda)(\Lambda_\epsilon(\lambda)\otimes I_n),
\end{equation}
considered as a matrix polynomial of grade $\epsilon+\eta+1$.
\end{theorem}

In the next section, we recall how to use block Kronecker pencils to construct strong linearizations of a prescribed matrix polynomial.

\subsection{Block Kronecker linearizations}
In this section, we are given an  $n\times n$ matrix polynomial $P(\lambda)$ of degree $d$, and our goal is to obtain strong linearizations for $P(\lambda)$ from the family of block Kronecker pencils.

First, we write $d=\epsilon+\eta+1$, for some nonnegative integers $\epsilon$ and $\eta$.
Then, by Theorem \ref{thm:kronecker pencils}, we see that in order for a block Kronecker pencil as in \eqref{eq:block Kronecker pencil} to be a strong linearization of $P(\lambda)$, it needs to satisfy
\begin{equation}\label{eq:P}
(\Lambda_\eta(\lambda)^T\otimes I_n)M(\lambda)(\Lambda_\epsilon(\lambda)\otimes I_n) = P(\lambda).
\end{equation}
This motivates the following definition.
\begin{definition}[block Kronecker linearization]
Given an $n\times n$ matrix polynomial $P(\lambda)$ as in \eqref{eq:poly} of degree $d$, an $(\epsilon,n,\eta,n)$-block Kronecker pencil as in \eqref{eq:block Kronecker pencil} is called a \emph{block Kronecker linearization of $P(\lambda)$} if $d=\epsilon+\eta+1$ and the pencil $M(\lambda)$ satisfies \eqref{eq:P}.
\end{definition}

The set of block Kronecker linearizations is never empty, because \eqref{eq:P} is consistent for each matrix polynomial $P(\lambda)$.
Indeed, it is easy to check that the pencil
\begin{equation}\label{eq:M0}
M_0(\lambda):=\begin{bmatrix}
\lambda A_d+A_{d-1} & A_{d-2} & \cdots & A_\eta \\
0 & \cdots & 0 & \vdots \\
\vdots & \ddots & \vdots & A_1 \\
0 & \cdots & 0 & A_0
\end{bmatrix}
\end{equation}
is a solution of \eqref{eq:P}.

The pencil $M_0(\lambda)$ in \eqref{eq:M0} is not the only solution of \eqref{eq:P}.
We present in Theorem \ref{thm:characterization of solutions} two equivalent characterizations of the set of solution of \eqref{eq:P}.
\begin{theorem}\label{thm:characterization of solutions}
Let $P(\lambda)$ as in \eqref{eq:poly} be an $n\times n$ matrix polynomial of degree $d$, and let $\epsilon$ and $\eta$ be nonnegative integers such that $\epsilon+\eta+1=d$.
Then, the following conditions are equivalent.
\begin{itemize}
\item[\rm (i)] The pencil $M(\lambda)$ satisfies \eqref{eq:P}.
\item[\rm (ii)] The pencil $M(\lambda)$ is of the form
\[
M(\lambda)=M_0(\lambda)+B(L_\epsilon(\lambda)\otimes I_n)+(L_\eta(\lambda)^T\otimes I_n)C,
\]
for some matrices $B\in\mathbb{C}^{(\eta+1)n\times \epsilon n}$ and $C\in\mathbb{C}^{\eta n\times(\epsilon+1)n}$, where $M_0(\lambda)$ has been defined in \eqref{eq:M0}.
\item[\rm (iii)] If we consider the pencil $M(\lambda)=\lambda M_1+M_0$ as an $(\eta+1)\times( \epsilon+1)$ block pencil with $n\times n$ block entries, denoted by $[M(\lambda)]_{ij} = \lambda [M_1]_{ij}+[M_0]_{ij}$, then the pencil $M(\lambda)$  satisfies the anti-diagonal sum conditions
\[
\sum_{i+j=d+2-k}[M_1]_{ij}+\sum_{i+j=d+1-k}[M_0]_{ij}=A_k, \quad \mbox{for }k=0,1,\hdots,d.
\]
\end{itemize}
\end{theorem}
\begin{proof}
This is just \cite[Theorem 5.9]{ell-ifications-DPV} with $\ell=1$.
\end{proof}

\begin{remark}
We recall that the Frobenius companion form \eqref{eq:Frob}, Fiedler pencils \cite{Fiedler} and generalized Fiedler pencils \cite{vectors} are, modulo permutations, examples of block Kronecker linearizations \cite{simplified,kronecker}.
\end{remark}

%

\subsection{Local right-sided factorizations for block Kronecker linearizations}
We show in this section that block Kronecker linearizations present factorizations as those introduced in Definition \ref{def:local fact}.

We begin  by introducing two rectangular block-Toeplitz matrix polynomials that play an important role in the following developments.
\begin{equation}\label{eq:Rk}
R_k(\lambda):=
\begin{bmatrix}
I_n & 0 & \cdots & 0 & 0 \\
\lambda I_n & \ddots & \ddots & \vdots & \vdots \\
\vdots & \ddots & I_n & 0 & \vdots \\
\lambda^{k-1}I_n & \cdots & \lambda I_n & I_n & 0 
\end{bmatrix}\in\mathbb{C}[\lambda]^{kn\times (k+1)n},
\end{equation}
and
\begin{equation}\label{eq:Sk}
S_k(\lambda):=
\begin{bmatrix}
0 & \lambda^{k-1}I_n & \lambda^{k-2}I_n & \cdots & I_n \\
\vdots & 0 & \lambda^{k-1}I_n & \ddots & \vdots \\
\vdots & \vdots & \ddots & \ddots & \lambda^{k-2}I_n \\
0 & 0 & \cdots & 0 & \lambda^{k-1}I_n
\end{bmatrix}\in\mathbb{C}[\lambda]^{kn\times (k+1)n}.
\end{equation}
For $k=0$, we follow the convention of defining both $R_0(\lambda)$ and $S_0(\lambda)$ as the empty matrix.

In Theorem \ref{thm:local fact Kronecker}, we establish two local right-sided factorizations for block Kronecker linearizations.
\begin{theorem}[local right-sided factorizations]\label{thm:local fact Kronecker}
Let $P(\lambda)$ be an $n\times n$ matrix polynomial as in \eqref{eq:poly} of degree $d$, and let $\mathcal{L}(\lambda)$  as in \eqref{eq:block Kronecker pencil} be a block Kronecker linearization of $P(\lambda)$.
Then, the following factorizations hold.
\begin{equation}\label{eq:right-sided1}
\mathcal{L}(\lambda) 
\begin{bmatrix}
\Lambda_\epsilon(\lambda)\otimes I_n \\
R_\eta(\lambda)M(\lambda)(\Lambda_\epsilon(\lambda)\otimes I_n)
\end{bmatrix} = e_{\eta+1}\otimes P(\lambda),
\end{equation}
and 
\begin{equation}\label{eq:right-sided2}
\mathcal{L}(\lambda) 
\begin{bmatrix}
\frac{1}{\lambda^{\epsilon}}(\Lambda_\epsilon(\lambda)\otimes I_n) \\
\frac{1}{\lambda^{d-1}}S_\eta(\lambda)M(\lambda)(\Lambda_\epsilon(\lambda)\otimes I_n)
\end{bmatrix} = \frac{1}{\lambda^{d-1}}e_{1}\otimes P(\lambda),
\end{equation}
where $R_k(\lambda)$ and $S_k(\lambda)$ are defined in \eqref{eq:Rk} and \eqref{eq:Sk}, respectively.
Moreover, the factorization \eqref{eq:right-sided1} is a right-sided factorization in $\mathbb{C}$, and the factorization \eqref{eq:right-sided2} is a right-sided factorization in $\mathbb{C}\setminus\{0\}$.
\end{theorem}
\begin{proof}
See \cite[Theorem 3.3]{conditioning Kronecker}.
\end{proof}

\section{Eigenvector error bounds for polynomial eigenvalue problems}\label{sec:error-bounds}
We consider in this section a linearization $L(\lambda)$ of a matrix polynomial $P(\lambda)$ satisfying properties P1 and P2, as explained in Section \ref{sec:Matrix Polynomials}, and an approximate eigenpair $(\widetilde{\lambda}_0,\widetilde{x})$ of an exact eigenpair $(\lambda_0,x)$ of $P(\lambda)$.
Our aim is to bound the acute angle between the approximate and exact eigenvectors $\widetilde{x}$ and $x$.
Recall that the local factorization \eqref{eq:one-sided-fact} implies that the vector $v=H(\lambda_0)x$ is an eigenvector of $L(\lambda)$ with eigenvalue $\lambda_0$.
Following \cite{TN}, we will considered the pair
\[
(\widetilde{\lambda}_0,\widetilde{v}:=H(\widetilde{\lambda}_0)\widetilde{x})
\]
as an approximate eigenpair of $L(\lambda)$.
By invoking the theory for GEPs in Section \ref{sec:GEP}, we can quantify $\sin\angle(v,\widetilde{v})$, because $v=H(\lambda_0)x$ is an exact eigenvector of $L(\lambda)$.
Then, combining this with Theorem \ref{thm:bound-angle}, we can bound the sine of the acute angle between the approximate and exact eigenvectors $\widetilde{x}$ and $x$ in terms of the residual norm $\|P(\widetilde{\lambda}_0)\widetilde{x}\|_2$ and the separation of $\widetilde{\lambda}_0$ to the set of eigenvalues of the linearization $L(\lambda)$ other than $\lambda_0$.

\begin{theorem}\label{thm:right-eig-bounds}
Let $P(\lambda)$ as in \eqref{eq:poly} be an $n\times n$ matrix polynomial of degree $d$, and let $L(\lambda)$ be a linearization of $P(\lambda)$ satisfying properties {\rm P1} and {\rm P2}.
Let $(\lambda_0,x)$ and $(\widetilde{\lambda}_0,\widetilde{x})$ be, respectively, exact and approximate eigenpairs of $P(\lambda)$. 
Assume $x$ and $\widetilde{x}$ are  unit vectors, and $\widetilde{\lambda}_0\in\Omega$.
Suppose the linerization $L(\lambda)=A-\lambda B$ has a generalized Schur form
\[
Q^*BZ=\begin{bmatrix} \alpha & b^* \\ 0 & B_1 \end{bmatrix} \quad \mbox{and} \quad 
Q^*AZ=\begin{bmatrix} \beta & a^* \\ 0 & A_1 \end{bmatrix} \quad \mbox{with} \quad \lambda_0 = \frac{\beta}{\alpha},
\] 
where $Q$ and $Z$ are unitary, and  $A_1$ and $B_1$ are upper triangular.
Then,
\begin{equation}\label{eq:main-bound1}
\sin\angle(x,\widetilde{x}) \leq \frac{\|g(\widetilde{\lambda}_0)\|_2\|P(\widetilde{\lambda}_0)\widetilde{x}\|_2}{\mathrm{sep}\,(\widetilde{\lambda}_0,(A_1,B_1))}.
\end{equation}
\end{theorem}
\begin{proof}
Let $v:=H(\lambda_0)x$ and $\widetilde{v}:=H(\mu_0)\widetilde{x}$, and let us partition the vectors $v$ and $\widetilde{v}$ into $d$ blocks of size $n$ each, denoted by $v_j$ and $\widetilde{v}_j$, respectively.
Recall that property P2 implies $v_i=x$ and $\widetilde{v}_i=\widetilde{x}$, for some $i\in\{ 1,\hdots,d \}$. 
Then, from Theorem \ref{thm:bound-angle}, we have 
\[
\sin\angle(x,\widetilde{x})\leq \min\{\| v \|_2,\| \widetilde{v}\|_2 \}\sin\angle(v,\widetilde{v})\leq \| \widetilde{v} \|_2 \sin\angle(v,\widetilde{v}).
\]
Since the pair $(\lambda_0,v)$ is an exact eigenpair of $L(\lambda)$, we obtain from Theorem \ref{thm:gap}
\[
 \sin\angle(v,\widetilde{v})\leq \frac{\| L(\widetilde{\lambda}_0)\widetilde{v} \|_2}{\|\widetilde{v}\|_2\,\mathrm{sep}\,(\widetilde{\lambda}_0,(A_1,B_1))}.
\]
From Lemma \ref{thm:residual norm}, we finally get
\begin{align*}
\sin\angle(x,\widetilde{x})\leq   \| \widetilde{v} \|_2 \sin\angle(v,\widetilde{v}) \leq & \frac{\| L(\widetilde{\lambda}_0)\widetilde{v} \|_2}{\mathrm{sep}\,(\widetilde{\lambda}_0,(A_1,B_1))} = 
 \frac{ \|g(\widetilde{\lambda}_0)\|_2 \| P(\widetilde{\lambda}_0)\widetilde{x}\|_2}{\mathrm{sep}\,(\widetilde{\lambda}_0,(A_1,B_1))},
\end{align*}
as we wanted to show.
\end{proof}

In the following section, we particularize Theorem \ref{thm:right-eig-bounds} to the case when $L(\lambda)$ is a block Kronecker linearization.

\section{Eigenvector error bounds for polynomial eigenvalue problems from block Kronecker companion linearizations}\label{sec:error-bounds-BKL}

 Nakatsukasa and Tisseur obtained  as a corollary of their theory the following eigenvector error bound for the Frobenius companion form \eqref{eq:Frob}.
\begin{theorem}[\cite{TN}]
Under the assumptions of Theorem \ref{thm:right-eig-bounds}, if the linearization $L(\lambda)=A-\lambda B$ is the Frobenius companion form \eqref{eq:Frob}, then
\begin{equation}\label{eq:bound for Frobenius form}
\sin\angle(x,\widetilde{x})\leq \frac{\|P(\widetilde{\lambda}_0)\widetilde{x}\|_2}{\sqrt{\sum_{i=0}^{d-1}|\widetilde{\lambda}_0|^{2i}}\cdot\mathrm{sep}(\widetilde{\lambda}_0,(A_1,B_1))}.
\end{equation}
\end{theorem}

Armed with the local right-sided factorizations \eqref{eq:right-sided1} and \eqref{eq:right-sided2}, we can apply Theorem \ref{thm:right-eig-bounds} to any block Kronecker linearization (including the Frobenius companion form).

\begin{theorem}\label{thm:bound Kronecker}
Under the assumptions of Theorem \ref{thm:right-eig-bounds}, if the linearization $L(\lambda)=A-\lambda B$ is  a block Kronecker linearization, then
\begin{equation}\label{eq:main-bound2}
\sin\angle(x,\widetilde{x}) \leq \frac{\|P(\widetilde{\lambda}_0)\widetilde{x}\|_2}{\max\left\{1,|\widetilde{\lambda}_0|^{d-1}\right\}\cdot \mathrm{sep}\,(\widetilde{\lambda}_0,(A_1,B_1))}.
\end{equation}
\end{theorem}
\begin{proof}
From Theorem \ref{thm:right-eig-bounds} and the right-sided factorization \eqref{eq:right-sided1}, we readily obtain the upper bound
\begin{equation}\label{eq:aux bound}
\sin\angle(x,\widetilde{x}) \leq \frac{\|P(\widetilde{\lambda}_0)\widetilde{x}\|_2}{\mathrm{sep}\,(\widetilde{\lambda}_0,(A_1,B_1))},
\end{equation}
for any pair $(\widetilde{\lambda}_0,\widetilde{x})$.
Now, assume $|\widetilde{\lambda}_0|\geq 1$. 
Since $\widetilde{\lambda}_0\neq 0$, from Theorem \ref{thm:right-eig-bounds} and the right-sided factorization \eqref{eq:right-sided2}, we  obtain the upper bound
\begin{equation}\label{eq:aux bound2}
\sin\angle(x,\widetilde{x}) \leq \frac{\|P(\widetilde{\lambda}_0)\widetilde{x}\|_2}{|\widetilde{\lambda}_0|^{d-1}\cdot \mathrm{sep}\,(\widetilde{\lambda}_0,(A_1,B_1))}.
\end{equation}
Combining \eqref{eq:aux bound} and \eqref{eq:aux bound2} gives the desired result.
\end{proof}

It is instructive to compare the Nakatsukasa and Tisseur upper bound for the Frobenius companion form
\[
\mathrm{Frobenius-bound} = \frac{\|P(\widetilde{\lambda}_0)\widetilde{x}\|_2}{\sqrt{\sum_{i=0}^{d-1}|\widetilde{\lambda}_0|^{2i}}\cdot\mathrm{sep}(\widetilde{\lambda}_0,(A_1,B_1))},
\]
with the upper bound in Theorem \ref{thm:bound Kronecker} in the case when $L(\lambda)$ is the Frobenius companion form
\[
\mathrm{Kronecker-bound} = \frac{\|P(\widetilde{\lambda}_0)\widetilde{x}\|_2}{\max\left\{1,|\widetilde{\lambda}_0|^{d-1}\right\}\cdot \mathrm{sep}\,(\widetilde{\lambda}_0,(A_1,B_1))}.
\]
Clearly, we have
\[
1 \leq \frac{\mathrm{Kronecker-bound}}{\mathrm{Frobenius-bound}} \leq \sqrt{d},
\]
and, so, our bound is within a $\sqrt{d}$ distance from the bound by Nakatsukasa and Tisseur.
Hence, it is natural to consider the upper bounds in Theorem \ref{thm:bound Kronecker}  as extensions of the Nakatsukasa and Tisseur upper bound \eqref{eq:bound for Frobenius form} for the Frobenius companion form to block Kronecker linearizations. 

\section{Numerical examples}\label{sec:experiments}

We illustrate the theory on some random matrix polynomials.
The experiments were performed in MATLAB 8, for which the unit roundoff is $2^{-53} \approx 10^{-16}$.
To compute the acute angle $\angle(x,\widetilde{x})$ and the separation \eqref{eq:def-sep}, we took as exact eigenvectors the ones computed in MATLAB's VPA arithmetic at 40 digit precision.
The $x$-axis in all our figures represents eigenvector index.
The eigenvectors are sorted always in increasing order of absolute value of their corresponding eigenvalues.

We construct two  random matrix polynomials $P_1(\lambda)=\sum_{i=0}^dA_i\lambda^i$ and $P_2(\lambda)=\sum_{i=0}^d B_i\lambda^i$, both with degree $d=5$ and size $n=10$. 
The polynomial $P_1(\lambda)$ has well-scaled eigenvalues, all of them of order $\mathcal{O}(1)$.
For $i=0,1,\hdots,5$, we construct the matrix coefficient $A_i$ with the MATLAB command
\[
A_i = {\tt randn(n)+sqrt(-1)*randn(n)};
\]
The polynomial $P_2(\lambda)$ has widely-scaled eigenvalues (the first fifth are $\mathcal{O}(10^{-4})$, the next two fifths are $\mathcal{O}(1)$, and the final two fifths are $O(10^4)$).
Its matrix coefficients $B_i$ are constructed as follows:
\begin{align*}
&B_0 = {\tt randn(n)+sqrt(-1)*randn(n);} \\
&B_1 = {\tt 1e4*(randn(n)+sqrt(-1)*randn(n));} \\ &B_2 ={\tt  1e-2*(randn(n)+sqrt(-1)*randn(n));} \\
&B_3 = {\tt 1e5*(randn(n)+sqrt(-1)*randn(n));} \\
&B_4 = {\tt randn(n)+sqrt(-1)*randn(n);} \\ &B_5  = {\tt 1e-1*(randn(n)+sqrt(-1)*randn(n));} 
\end{align*}
Then, the matrix polynomials $P_1(\lambda)$ and $P_2(\lambda)$ are scaled so that $\max_{i=0:d}\{\|A_i\|_2\}=\max_{i=0:d}\{\|B_i\|_2\}=1$.

The goal of the experiment is to show that the acute angle bound in Theorem \ref{thm:bound Kronecker} is experimentally tight, regardless of the linearization we use.
To verify this claim, we apply the bound \eqref{eq:main-bound2} to three block Kronecker companion linearizations: the Frobenius companion form
\begin{equation}\label{eq:expL1}
L_1(\lambda) =\begin{bmatrix}
A_4 & A_3 & A_2 & A_1 & A_0 \\
-I_n & 0 & 0 & 0 & 0 \\
0 & -I_n & 0 & 0 & 0 \\
0 & 0 & -I_n & 0 & 0 \\
0 & 0 & 0 & -I_n & 0
\end{bmatrix}-
\lambda \begin{bmatrix}
-A_5 & 0 & 0 & 0 & 0 \\
0 & -I_n & 0 & 0 & 0 \\
0 & 0 & -I_n & 0 & 0 \\
0 & 0 & 0 & -I_n & 0 \\
0 & 0 & 0 & 0 & -I_n
\end{bmatrix},
\end{equation}
the (permuted) Fiedler pencil
\begin{equation}\label{eq:expL2}
L_2(\lambda)=\begin{bmatrix}
A_4 & A_3 & A_2 & A_1 & -I_n \\
0 & 0 & 0 & A_0 & 0 \\
-I_n & 0 & 0 & 0 & 0 \\
0 & -I_n & 0 & 0 & 0 \\
0 & 0 & -I_n & 0 & 0 \\
\end{bmatrix}-\lambda\begin{bmatrix}
-A_5 & 0 & 0 & 0 & 0 \\
0 & 0 & 0 & 0 & - I_n\\
0 & -I_n & 0 & 0 & 0 \\
0 & 0 & -I_n & 0 & 0 \\
0 & 0 & 0 & -I_n & 0
\end{bmatrix},
\end{equation}
and the (permuted) generalized Fiedler pencil
\begin{equation}\label{eq:expL3}
L_3(\lambda) - \begin{bmatrix}
A_4 & 0 & 0 & -I_n & 0 \\
0 & A_2 & 0 & 0 & -I_n \\
0 & 0 & A_0 & 0 & 0 \\
-I_n & 0 & 0 & 0 & 0 \\
0 & -I_n & 0 & 0 & 0 
\end{bmatrix}-\lambda \begin{bmatrix}
-A_5 & 0 & 0 & 0 & 0 \\
0 & -A_3 & 0 & -I_n & 0 \\
0 & 0 & -A_1 & 0 & -I_n \\
0 & -I_n & 0 & 0 & 0 \\
0 & 0 & -I_n & 0 & 0
\end{bmatrix}.
\end{equation}

In Figures \ref{fig:L1}, \ref{fig:L2} and \ref{fig:L3} we plot the eigenvector errors (measured by the sine of the acute angle between computed and exact eigenvectors \eqref{eq:def-angle}) and the upper bound in Theorem \ref{thm:bound Kronecker}, for the polynomial $P_1(\lambda)$ (upper figures) and the polynomial $P_2(\lambda)$ (lower figures).
We observe that our bound was tight in all the experiments, off only by one order of magnitude.

\begin{figure}[h]
\centering
\includegraphics[height=6.5cm, width=13cm]{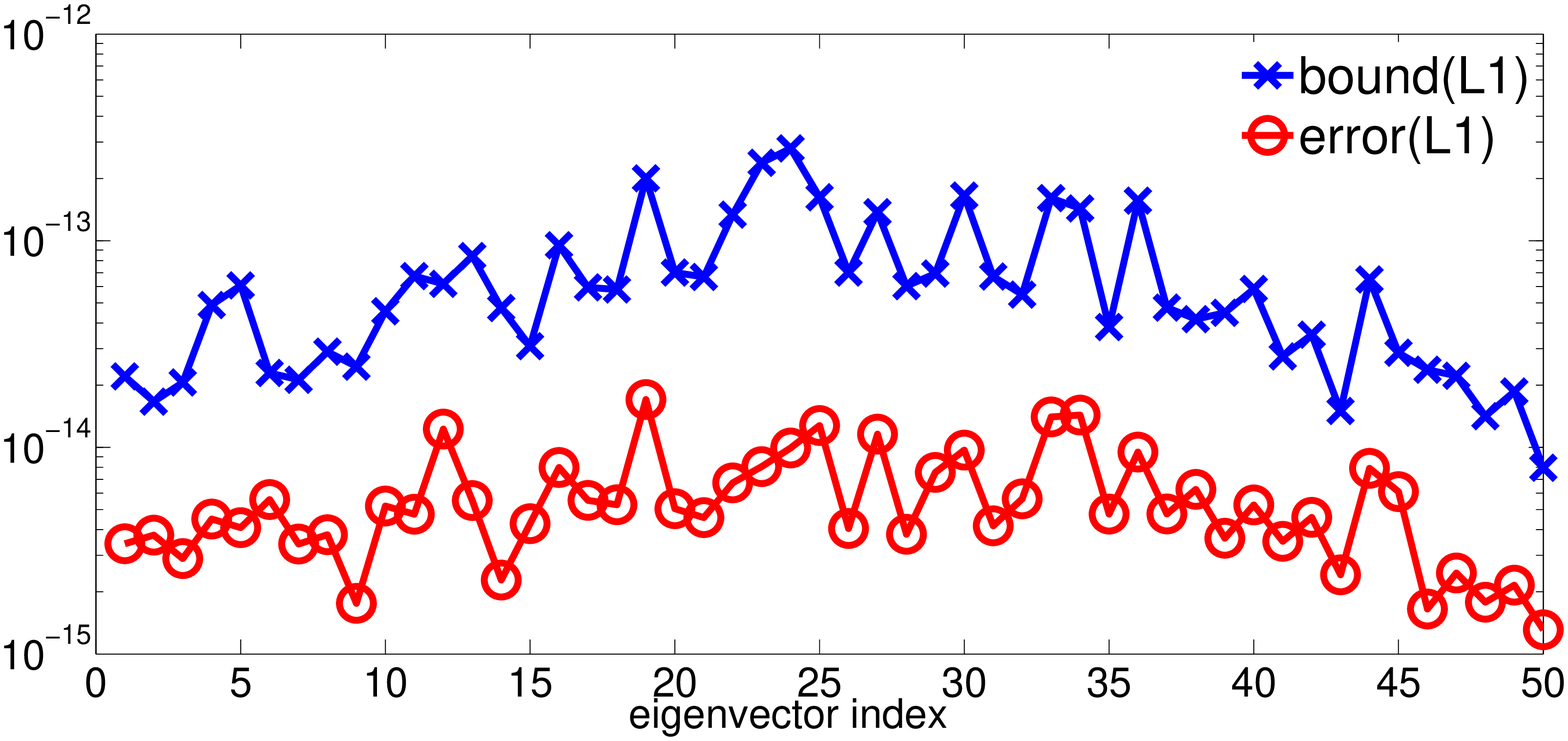}\\
\includegraphics[height=6.5cm, width=13cm]{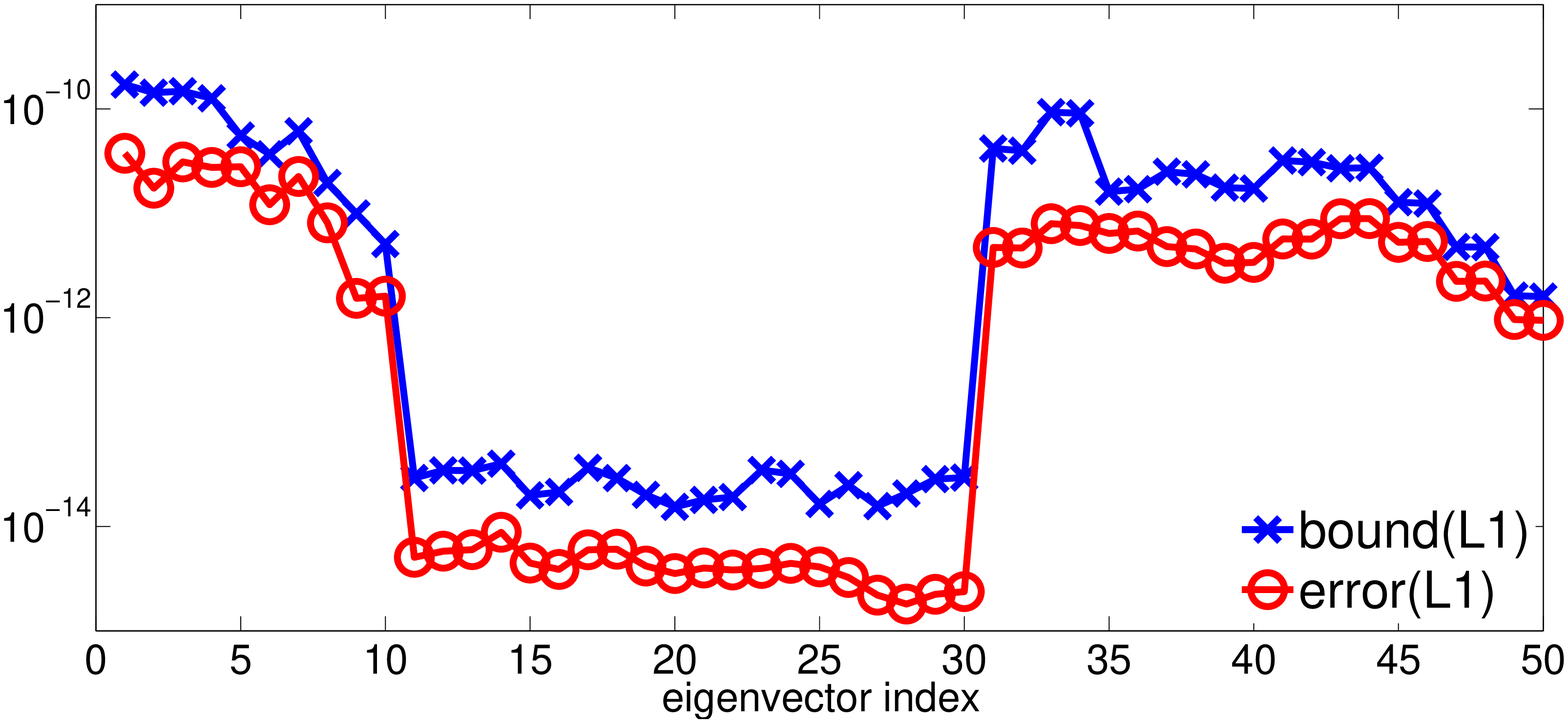}
\caption{Eigenvector error $\sin\angle(x,\widetilde{x})$ and the upper bound \eqref{eq:main-bound2} for the polynomials $P_1(\lambda)$ (upper figure) and $P_2(\lambda)$ (lower figure).
The linearization $L_1(\lambda)$ is as in \eqref{eq:expL1}.}
\label{fig:L1} 
\end{figure}

\begin{figure}[h]
\centering
\includegraphics[height=6.5cm, width=13cm]{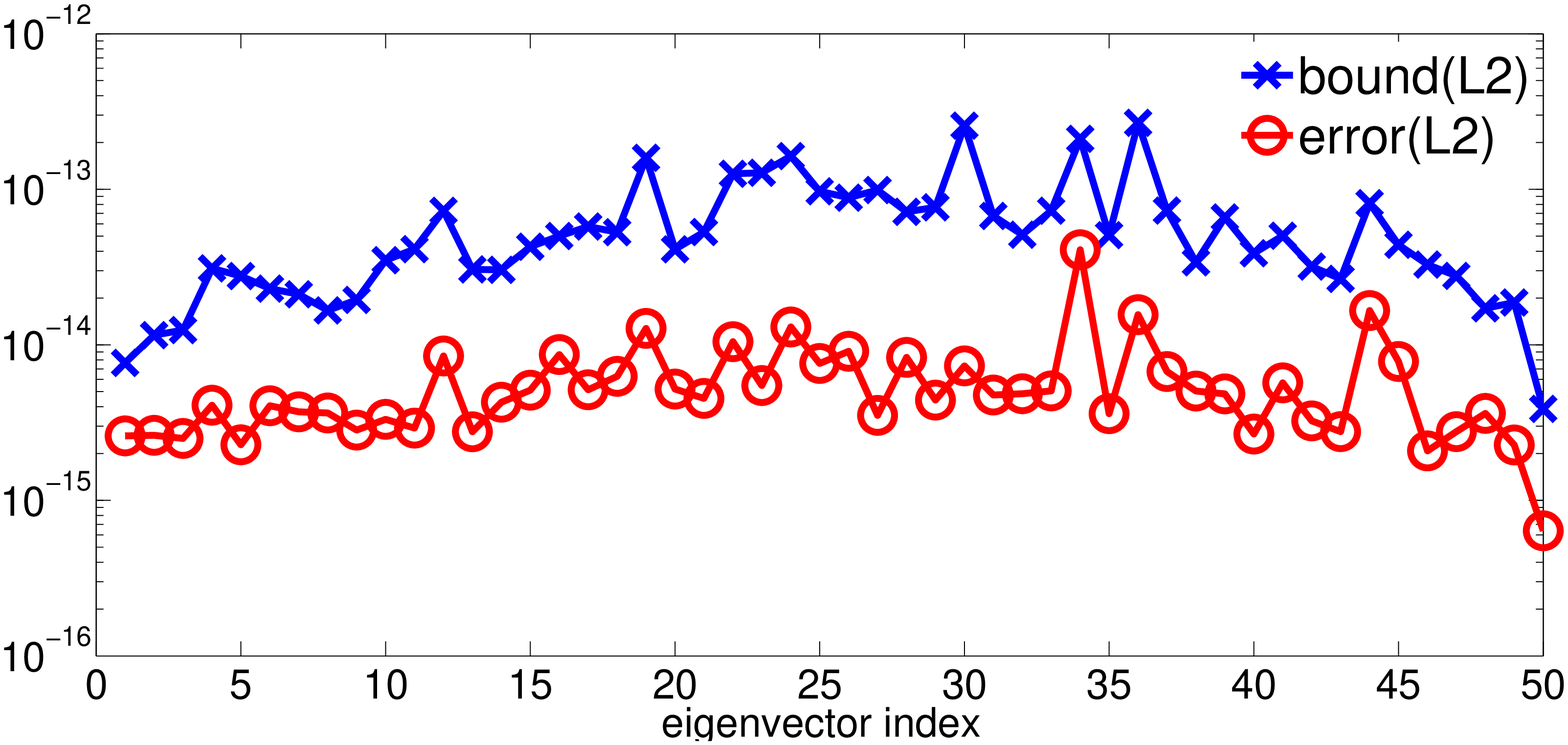} \\
\includegraphics[height=6.5cm, width=13cm]{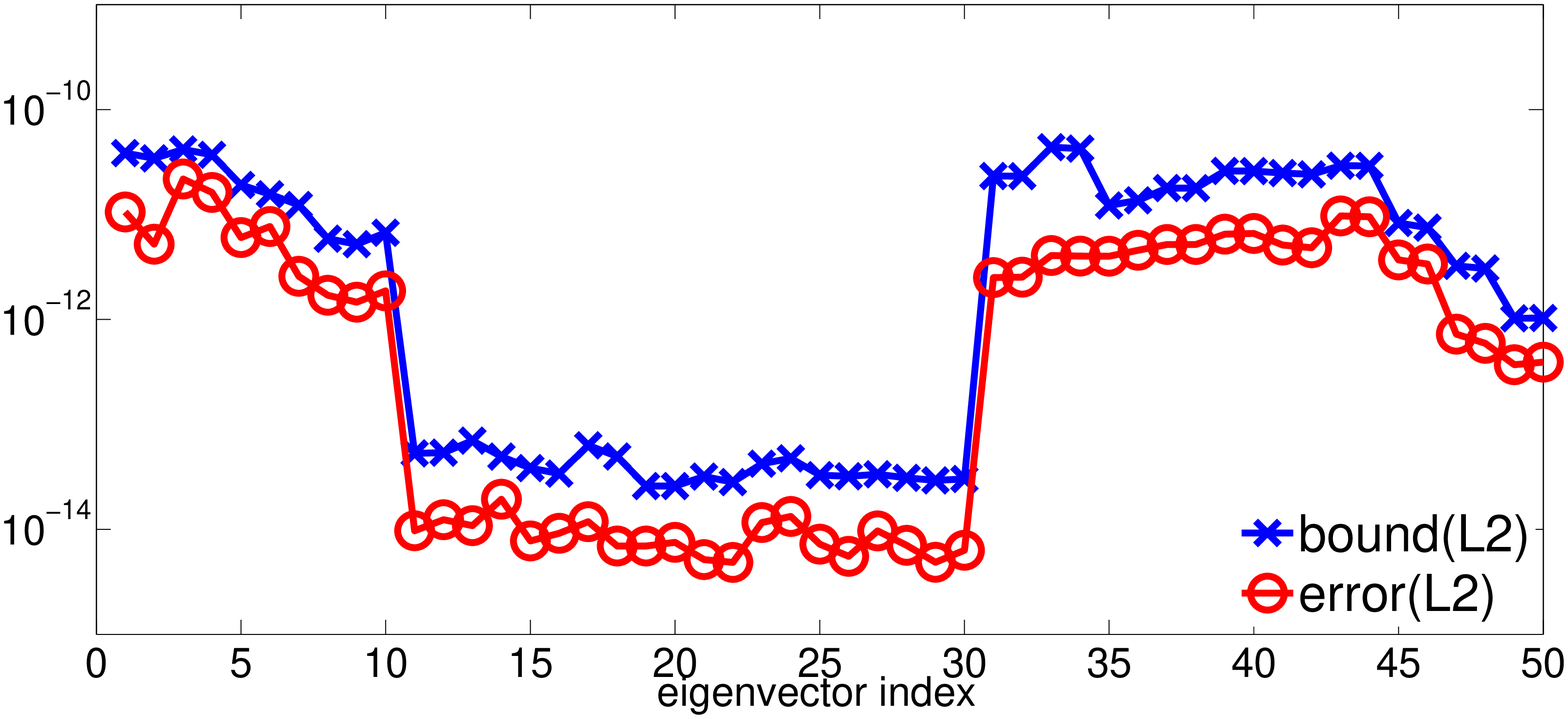} 
\caption{Eigenvector error $\sin\angle(x,\widetilde{x})$ and the upper bound \eqref{eq:main-bound2} for the polynomials $P_1(\lambda)$ (upper figure) and $P_2(\lambda)$ (lower figure).
The linearization $L_2(\lambda)$ is as in \eqref{eq:expL2}.} 
\label{fig:L2} 
\end{figure}

\begin{figure}[h]
\centering
\includegraphics[height=6.5cm, width=13cm]{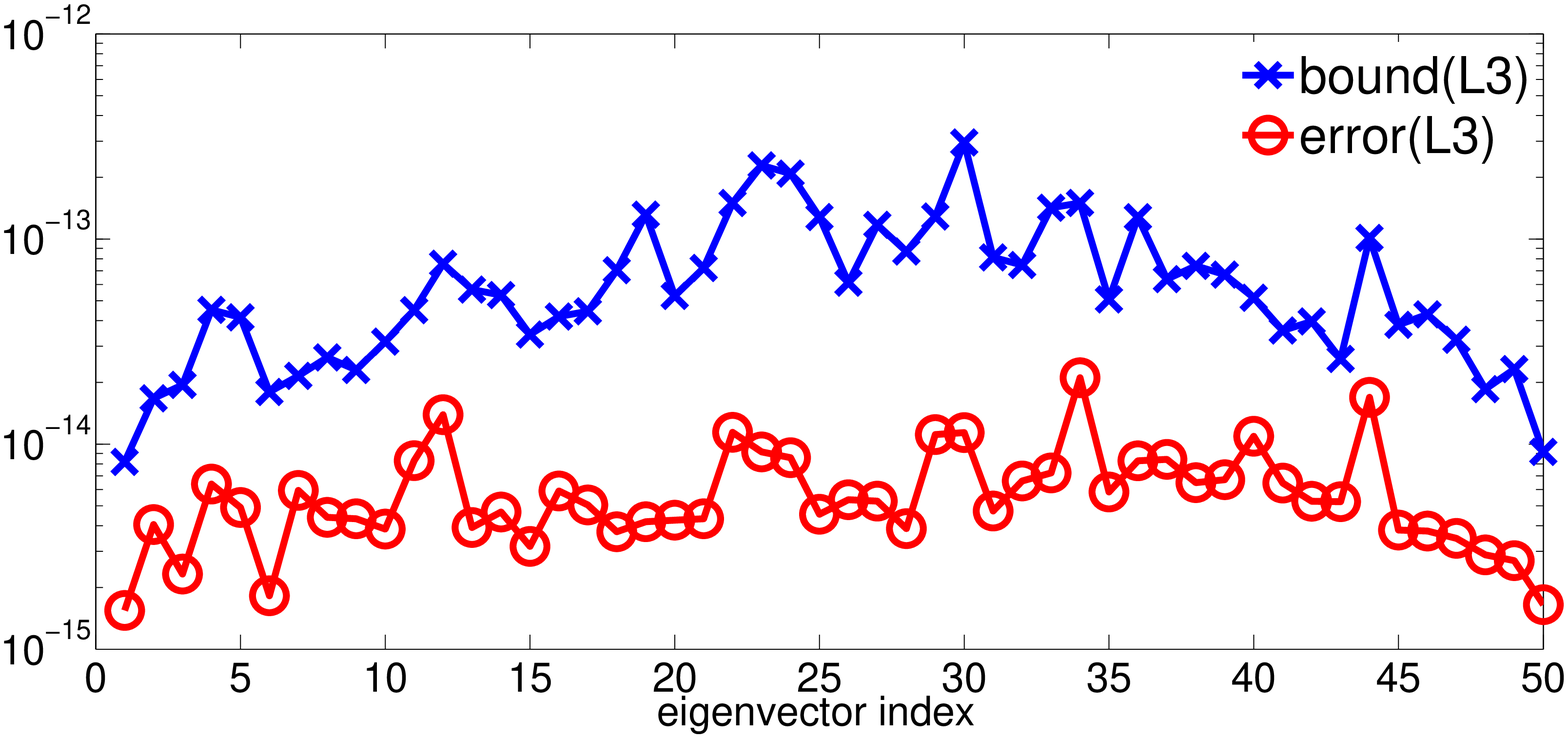} \\
\includegraphics[height=6.5cm, width=13cm]{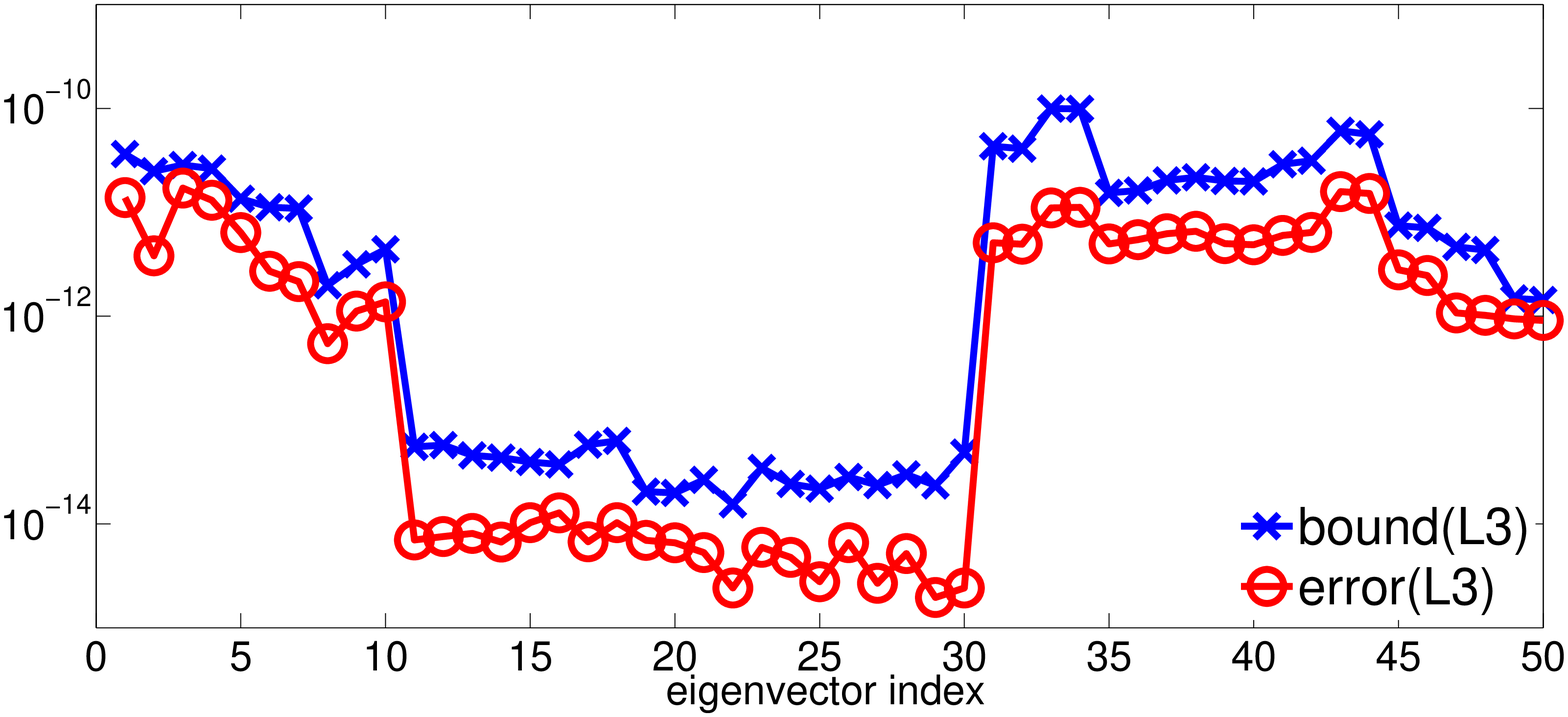}
\caption{Eigenvector error $\sin\angle(x,\widetilde{x})$ and the upper bound \eqref{eq:main-bound2} for the polynomials $P_1(\lambda)$ (upper figure) and $P_2(\lambda)$ (lower figure).
The linearization $L_3(\lambda)$ is as in \eqref{eq:expL3}.}
\label{fig:L3} 
\end{figure}

\end{document}